\documentclass{amsart}
\usepackage{amsmath,amssymb}
\theoremstyle{plain}
\newtheorem{thm}{Theorem}[section]
\newtheorem{lem}[thm]{Lemma}
\newtheorem{cor}[thm]{Corollary}

\theoremstyle{definition}
\newtheorem{defn}[thm]{Definition}

\newtheorem{rmk}[thm]{Remark}

\usepackage{graphicx,epsfig}

\title[The piecewise disjoint arc-disk property]{Detecting codimension one manifold factors with the piecewise disjoint arc-disk property and related properties}

\author[D. M. Halverson]{Denise M. Halverson}
\address{Department of Mathematics, Brigham Young University, Provo, UT 84602}
\email{halverson@math.byu.edu}

\author[D. Repov\v s]{Du\v san Repov\v s}
\address{Faculty of Education, and Faculty of Mathematics and Physics,
University of Ljubljana,
P.O. Box 2964,
Ljubljana, Slovenia 1001}
\email{dusan.repovs@guest.arnes.si}

\begin{document}

\begin{abstract}
We show that all finite-dimensional resolvable generalized manifolds with the piecewise disjoint arc-disk property are codimension one mani\-fold factors. We then show how the piecewise disjoint arc-disk property and other general position properties that detect codimension one manifold factors are related.  We also note that in every example presently known to the authors of a codimension one manifold factor of dimension $n\geq 4$ determined by general position properties, the piecewise disjoint arc-disk property is satisfied.
\end{abstract}

\date{\today}

\keywords{Piecewise disjoint arc-disk property,  general position, codimension one manifold factor, Generalized Moore Problem, totally wild flow, ghastly generalized manifold, plentiful 2-manifolds property, $0$-stitched disks, $\delta$-fractured maps, fractured maps property, crinkled ribbons property, fuzzy ribbons property, disjoint homotopies property, disjoint topographies property, disjoint concordances}

\subjclass[2010]{Primary 57N15, 57N75; Secondary 57P99, 53C70}

\maketitle

\section{Introduction}

A space $X$ is said to be a \emph{codimension one manifold factor}
provided that $X \times \mathbb{R}$ is a topological manifold.  The Product With a Line Problem is a long standing unsolved problem which asks whether or not all resolvable
generalized manifolds are codimension one manifold factors \cite{Daverman S, Bertinoro, Moore1, Moore2}. The purpose of this paper is twofold:  (1) to introduce a new unifying general position property, called the \emph{piecewise disjoint arc-disk property}, and its $1$-complex analogue, the \emph{piecewise disjoint arc-disk property*}; and (2)  demonstrate how the various general position properties known to detect codimension one manifold factors are related.

The piecewise disjoint arc-disk property is a general position property that captures the essence of why spaces arising from certain generalized constructions are codimension one manifold factors. The main result of this paper is the following theorem:

\begin{thm} \label{main thm}
If $X$ is a resolvable generalized $n$-manifold that satisfies the piecewise disjoint arc-disk property, then $X \times \mathbb{R}$ is an $(n+1)$-manifold.
\end{thm}

\noindent We actually detail the proof of the theorem in the case that the piecewise disjoint arc-disk property is replaced with the piecewise disjoint arc-disk property*, a slightly stronger property immediately implying piecewise disjoint arc-disk property.  But we introduce both properties to more thoroughly delineate the relationships with other general position properties.

The importance of the piecewise disjoint arc-disk property can be seen from its unifying perspective.  In all examples of codimension one manifold factors detected by general position properties currently known to the authors, the underlying nature inherent to the piecewise disjoint arc-disk property has provided the needed utility to demonstrate that a decomposition space of dimension $n \geq 4$ is a codimension one manifold factor.

In this paper, we will also illustrate how the piecewise disjoint arc-disk property, the piecewise disjoint arc-disk property*,  and other general position properties used in the detection of codimension one manifolds factors are related. We will introduce definitions for (a) a modified version of $\delta$-fractured maps, (b) the $\delta$-fractured maps property with respect to the modified definition, (c) the closed $0$-stitched disks property, and (d) the strong fuzzy ribbons property.   We will demonstrate that in the case of resolvable generalized manifolds: (1) the modified $\delta$-fractured maps property implies the disjoint homotopies property, (2) the closed $0$-stitched disks property implies the $\delta$-fractured maps property, (3)  the piecewise disjoint arc-disk property is equivalent to the $\delta$-fractured maps property, and (4) the piecewise disjoint arc-disk property* is equivalent to the strong fuzzy ribbons property.  We will also note several other implications that have either been previously proven or are fairly straightforward.

The importance of general position properties in detecting  codimension one manifold factors of dimension $n\geq 4$ is derived from the role of the disjoint disks property in characterizing manifolds of dimension $n \geq 5$ \cite{Daverman S, Daverman book,  Daverman-Halverson, Edwards 2}.   General position properties that are effective in detecting codimension one manifold factors of dimension $n\geq 4$ can be found in \cite{Daverman book, Daverman-Halverson 2, Halverson 1, Halverson 2, Halverson 3, HaRe2}.    A more general and extensive discussion of various types of general position properties can also be found in a recent study by Banakh and Valov \cite{Banakh-Valov}.

\section{Preliminaries}

Throughout this paper, we assume that spaces are finite dimensional.  We begin with some basic definitions and notation.

%In particular, in this section we aim to define what is meant by a resolvable generalized manifold.  We also make a few other needed notations.

A compact subset $C$ of a space $X$ is said to be \emph{cell-like} if for each neighborhood $U$ of $C$ in $X$, $C$ can be contracted  to a point in $U$ \cite{MiRe}.  A space $X$ is said to be \emph{resolvable} if there is a manifold $M$ and a proper surjective map $f:M \to X$ so that for each $x \in X$, $f^{-1}(x)$ is cell-like.

Finite dimensional resolvable spaces are known to be ANR's (i.e., locally contractible, locally compact, separable metric spaces) \cite{MiRe}.  The following theorems illustrate the useful extension properties that ANR's possess, which will be applied freely in this paper:

\begin{thm} [Homotopy Extension Theorem]

Suppose that $f:Y \to X$ is a continuous map where $Y$ is a metric space and $X$ is an ANR, $Z$ is a compact subset of $Y$ and $\epsilon > 0$. Then there exists $\delta > 0$ such that each $g_Z:Z \to X$ which is $\delta$-close to  $f|_Z$ extends to a map $g:Y \to X$ so that $g$ is $\epsilon$-homotopic to $f$. In particular, for any open set $U$ such that $Z \subset U \subset Y$, there is a homotopy $H:Y \times I \to X$ so that
\begin{enumerate}
    \item $H_0 = f$ and $H_1 = g$;
    \item $g|_Z = g_Z$;
    \item $H_t|_{Y-U} = f|_{Y-U}$ for all $t \in I$; and
    \item $diam(H(y \times I)) < \epsilon$ for all $y \in Y$.
\end{enumerate}
\label{HET}

\end{thm}

\begin{cor}[Map Extension Theorem] Suppose that $f:Y \to X$ is a continuous map where $Y$ is a metric space and $X$ is an ANR, $Z$ is a compact subset of $Y$ and $\epsilon > 0$. Then there exists $\delta > 0$ such that each $g_Z:Z \to X$ which is $\delta$-close to  $f|_Z$ extends to $g:Y \to X$ so that $\rho(f,g) < \epsilon$. \label{MET}
\end{cor}

A set $Z \subset X$ is said to be $0$-LCC embedded in $X$ if for every  $z \in Z$, each neighborhood $U \subset X$ of $z$ contains a neighborhood $V \subset X$ of $z$ such that any two points in $V-Z$ are connected by a path  in $U-Z$.

A point $x \in X$ is said to be $1$-LCC embedded in $X$ if every neighborhood $U \subset X$ of $x$ contains a neighborhood $V \subset X$ of $x$ such that any map $f: \partial D^2 \to V-\{x\}$ can be extended to a map $f: D^2 \to U-\{x\}$.

We say that a homotopy $f:Z \times [a,b]\to X$ \emph{realizes} $g:Z\times [c,d] \to X$ if $f(x,t) = g(x, \gamma(t))$ for $t \in [a,b]$, where $\gamma:[a,b]\to [c,d]$ is the linear map from the interval $[a,b]$ onto the interval $[c,d]$ such that $\gamma(a)=c$ and $\gamma(b)=d$.

Suppose $f_i: Z \times I \to X$ for $i=1,\ldots, N$, where $f_i(x,1)= f_{i+1}(x,0)$ for $i=1,\ldots, N-1$.  We say that the \emph{adjunction} of $f_1, f_2, \ldots, f_N$, denoted $f=f_1\cdot f_2 \cdot \ldots \cdot f_N$, is the homotopy $f: Z \times I \to X$ so that $f|_{Z \times \left[\frac{i-1}{N},\frac{i}{N} \right]} \text{ realizes }f_i$ for $i=1,\ldots, N$.

\section{General Position Properties}

A space $X$ is said to have the $(k,m)$-DDP provided that any two maps $f: D^k \to X$ and $g: D^m \to X$ can be approximated arbitrarily closely by maps with disjoint images.  The $(k,m)$-DDP is satisfied by $n$-manifolds whenever $n \geq k+m+1$ \cite{rourke-sanderson}.  The $(1,1)$-DDP is more commonly called the \emph{disjoint arcs property (DAP)}.  A resolvable generalized manifold of dimension $n \geq 3$ has the DAP (see Proposition 26.3 of \cite{Daverman book}), but no other general position properties for $k+m \geq 2$ need be satisfied.  Even the $(0,2)$-DDP fails to hold in the famous Daverman-Walsh ghastly spaces, which are resolvable generalized manifolds of dimension $n\geq3$ that contain no embedded $2$-cells.  In these spaces every singular disk necessarily contains an open set \cite{Daverman-Walsh}.

Several techniques have by now been developed for detecting codimension one manifold factors of dimension $n \geq 4$.  In particular, a resolvable generalized manifold $X$ of dimension $n \geq 4$ is known to be a codimension one manifold factor in the case it has one of the following general position properties: the disjoint arc-disk property \cite{Daverman 1}, the disjoint homotopies property \cite{Halverson 1}, or the disjoint topographies (or disjoint concordance) property \cite{Daverman-Halverson 2, HaRe2}. The disjoint arc-disk property, satisfied by manifolds of dimension $n\geq 4$, is the most natural first guess as a general position property to detect codimension one manifold factors.  However, although sufficient, it is not necessary (for examples, see \cite{Cannon-Daverman2, Daverman-Walsh,Halverson 1}). On the other hand, the disjoint topographies (or disjoint concordance) property is a necessary and sufficient condition for resolvable spaces of dimension $n \geq 4$ to be codimension one manifold factors \cite{Daverman-Halverson 2, HaRe2}.  It is still unknown whether or not the disjoint homotopies property likewise provides such a characterization.

There are also several related general position properties that fall into subclasses of these properties. For example, spaces that have the plentiful $2$-manifolds property \cite{Halverson 1}, the $0$-stitched disks properties \cite{Halverson 3}, or for which the method of $\delta$-fractured maps can be applied \cite{Halverson 2}, all have the disjoint homotopies property.  The crinkled ribbons property, the twisted crinkled ribbons property, and the fuzzy ribbons property all imply the disjoint topographies property \cite{HaRe2}.

We now further detail each property with related properties and relevant results.

\subsection{The Disjoint Arc-Disk Property}

Let $I$ denote the unit interval and $D^2$ denote a disk.

\begin{defn}
A space $X$ is said to have the \emph{disjoint arc-disk property (DADP)} provided that any two maps $\alpha: I \to X$ and $f: D^2 \to X$ can be approximated arbitrarily closely by maps with disjoint images.
\end{defn}

%\noindent The essence of the following theorem is proved in
%\cite{Daverman 1}:

\begin{thm}  \label{DADP thm}  \cite{Daverman 1} Every resolvable generalized manifold having DADP is a codimension one manifold factor. \end{thm}

\subsection{The Disjoint Homotopies Property and Related Properties}

Let both $D$ and $I$ denote the unit interval $[0,1]$.

\begin{defn}
A space $X$ has the \emph{disjoint homotopies property (DHP)} if any two path homotopies $f,g:D \times I \to X$ can be approximated arbitrarily closely by homotopies $f',g':D \times I \to X$ so that $f'_{t}(D) \cap g'_{t}(D) = \emptyset$ for all $t \in I$.
\end{defn}

\begin{thm}
\cite{Halverson 1} Every resolvable generalized $n$-manifold having the DHP  is a codimension one manifold factor.
\end{thm}

\begin{defn}
A space $X$ has the \emph{plentiful $2$-manifolds property (P2MP)} if each path $\alpha: I \to X$ can be approximated arbitrarily closely by a path $\alpha': I \to N \subset X$ where $N$ is a $2$-manifold
embedded in $X$.
\end{defn}

\begin{thm} \cite{Halverson 1}
Every resolvable generalized $n$-manifold, $n \geq 4$, having the P2MP, satisfies the DHP, and hence is a codimension one manifold factor.
\end{thm}

\begin{defn} \label{frac def}
A map $f:D \times I \to X$ is said to be \emph{$\delta$-fractured} over a map $g:D \times I \to X$ if there are pairwise disjoint balls $B_1, B_2, \ldots, B_m$ in $D \times I$ such that for each $i \in \{1, \ldots, m\}$:
    \begin{enumerate}
        \item $diam(B_i) < \delta$;
        \item $f^{-1}(im(g)) \subset \bigcup_{i=1}^m int(B_i)$; and
        \item $diam(g^{-1}(f(B_i))) < \delta$.
    \end{enumerate}
\end{defn}

\begin{thm} \cite{Halverson 2}
If $X$ is a resolvable generalized $n$-manifold, $n \geq 4$, with the property that for an arbitrary homotopy $f:D \times I \to X$, constant homotopy $g: D\times I \to X$, and $\delta>0$, there are approximations $f'$ of $f$ and $g'$ of $g$ such that $f'$ is $\delta$-fractured over $g'$, then $X$ has the DHP, and hence $X$ is a codimension one manifold factor.
\end{thm}

In fact, upon closer inspection only a weaker modified version of the $\delta$-fractured maps property is required to detect codimension one manifold factors, which we will now define and prove here:

\begin{defn}[{\bf Modified Definition \ref{frac def}}]
A map $f:D \times I \to X$ is said to be \emph{$\delta$-fractured} over a map $g:D \times I \to X$ if there are pairwise disjoint balls $B_1, B_2, \ldots, B_m$ in $D \times I$ such that:
    \begin{enumerate}
        \item $diam(B_i) < \delta$;
        \item $f^{-1}(im(g)) \subset \bigcup_{i=1}^m int(B_i)$; and
        \item $p\circ g^{-1}\circ f(B_i) \ne I$;
    \end{enumerate}
where $p:D \times I \to I$ is the natural projection map.
\end{defn}

\begin{defn}
A space $X$ is said to have the \emph{$\delta$-fractured maps property ($\delta$-FMP)} provided that for any path homotopy $f:D \times I \to X$, constant homotopy $g:D \times I \to X$, and $\delta>0$, there are approximations $f'$ of $f$ and $g'$ of $g$ so that $f$ is $\delta$-fractured (modified version) over $g$.
\end{defn}

\begin{thm} \label{FM thm}
Every ANR having the $\delta$-FMP  satisfies the DHP.
\end{thm}

\begin{proof}
Let $X$ be an ANR with the $\delta$-FMP.  Note that $X$ has the DAP (this can be simply verified by applying the $\delta$-FMP to two constant path homotopies).

Let $f,g:D \times I \to X$ where $g$ is a constant homotopy. Applying the DAP, we may assume, without loss of generality, that $f(D \times \mathbb{Q}^*) \cap g(D \times \mathbb{Q}^*) = \emptyset$, where $\mathbb{Q}^* = \mathbb{Q} \cap I$.  Let $\epsilon >0$.  Choose an $N$ sufficiently large so that $diam \left(f \left(\{x\} \times \left[\frac{i-1}{2^N}, \frac{i}{2^N}\right]\right) \right) < \frac{\epsilon}4$ and $diam\left(g \left(\{x\} \times \left[\frac{i-1}{2^N}, \frac{i}{2^N}\right]\right) \right) < \frac{\epsilon}4$ for all $x \in D$  and $i=1,\ldots, 2^N$. Define
\begin{align*}
&f_i=f|_{D\times\left[\frac{i-1}{2^N}, \frac{i}{2^N}\right]}, \\
&\lambda_i:D \times I \to X \text{ is the constant homotopy }\lambda_i(x,t) = f_{\frac{i}{2^N}}(x),  \\&g_i=g|_{D\times\left[\frac{i-1}{2^N}, \frac{i}{2^N}\right]}, \text{ and } \\
&\gamma_i:D \times I \to X \text{ is the constant homotopy } \gamma_i(x,t) = g_{\frac{i-1}{2^N}}(x).
\end{align*}
Note that the adjunction maps $\tilde{f} = f_1\cdot \gamma_1 \cdot \ldots \cdot f_{2^N} \cdot \gamma_{2^N}$ and $\tilde{g} = \lambda_1\cdot g_1 \cdot \ldots \cdot \lambda_{2^N} \cdot g_{2^N}$ are $\epsilon/4$-approximations of $f$ and $g$, respectively.

  Choose $\xi$ so that $0<\xi< \epsilon/4$ and $\xi < \frac 12 dist\left(f_{\frac{i}{2^N}}(D), g_{\frac{i}{2^N}}(D)\right)$ for all $i=0, \ldots, 2^N$.   Then $\xi < \frac 12 dist\left((f_i)_e(D), (\gamma_i)_e(D)\right)$ and $\xi < \frac 12 dist\left((\lambda_i)_e(D), (g_i)_e(D)\right)$  for $e=0,1$ and all $i=1, \ldots, 2^N$.  Choose $\zeta>0$ so that $\zeta$-approximations of $f_{\frac{i}{2^N}}$ or $g_{\frac{i}{2^N}}$ are $\xi$-homotopic to $f_{\frac{i}{2^N}}$ or $g_{\frac{i}{2^N}}$, respectively.  Note that necessarily $\xi < \epsilon/4$.

Let $f'_i$ and $\gamma'_i$ be $\zeta$-approximations of $f_i$ and $\gamma_i$, respectively, so that $f'_i$ is $\delta$-fractured (modified version) over $\gamma'_i$, where $\delta>0$ is sufficiently small so that if $B'_1, B'_2, \ldots, B_m' \subset D \times I$ are the balls in the domain of $f_i'$, promised by the $\delta$-fractured maps condition, then $diam(f_i(B'_j))< \epsilon/2$.  Let $p: D \times I \to I$ be the natural projection map.  Let $\psi_i: D \times I \to D \times I$ be a homeomorphism taking each ball $B'_j$ to a ball $B_j$  to such that $p(B_j) \cap p(B_k) = \emptyset$ if $j \ne k$ and $f_i'' = f_i' \circ \psi_i$ is an $\epsilon/2$-approximation of $f_i'$.  The map $\psi_i$ can be defined by its inverse $\psi_i^{-1}$.  The map $\psi_i^{-1}$ is obtained by selecting a point $(x_j,t_j) \in int(B_j')$ for each $j=1, \ldots, m$ so that $t_j \ne t_k$ if $j \ne k$ and a small ball neighborhood $U_j$ of $B_j'$ so that the $diam(f_i'(U_j))<\epsilon/2$ and $U_j \cap U_k =\emptyset$ if $j \ne k$.  The ball $B_j'$ is compressed within $U_j$ to a very tiny ball neighborhood $B_j$ of $(x_j,t_j)$, so that $p(B_j) \cap p(B_k) = \emptyset$ if $j \ne k$.  Note that $f''_i$  is $\delta$-fractured over $\gamma'_i$ with respect to the balls $B_1, B_2, \ldots, B_m$.

Now we demonstrate for a fixed $i$ how to  reparameterize $\lambda'_i$ to obtain approximation $\lambda''_i$ such that $f''_i$ and $\lambda''_i$ are disjoint homotopies, as follows:  Given $B_1, B_2, \ldots, B_m$ to be the balls in the domain of $f''_i$ obtained above, such that $\psi_i(B_j)=B_j'$ and $p(B_j) \cap p(B_k) =\emptyset$ if $j \ne k$,  partition $I$ into a collection of subintervals with nonempty interiors, $\mathcal{J} = \{J_1, J_2, \ldots, J_r\}$, so that
\begin{enumerate}
\item $J_1 \leq J_2 \leq \ldots \leq J_r$,
\item $J_1$ and $J_r$ do not contain any of the projection sets $p(B_j)$,
\item for $j = 1, \ldots, m$, $p(B_j) \subset int(J)$ for some $J \in \mathcal{J}$,
\item if $p(B_j) \subset J \in \mathcal{J}$, then $p(B_k) \nsubseteq J$ when $j\ne k$, and
\item if $j \ne k$, there is at least one interval $J$  between the intervals containing $p(B_j)$ and $p(B_k)$ such that $J$ contains no projection set $p(B_l)$.
\end{enumerate}
For each $J_k \in \mathcal{J}$, we define a parameter value $\tau_k$ as follows:  Let $\tau_0=0$ and $\tau_{r}=1$.  For $0<k<r$, if $J_k$ or $J_{k+1}$ contains $p(B_j)$ then let $\tau_k$ be a value $t_j$ where $f''_i(B_j) \cap \gamma'_i(D \times \{t_j\}) = \emptyset$.  Such a value is guaranteed because $f''$ is $\delta$-fractured over $\gamma'$.  Otherwise let $\tau_k=0$.

We now define the homotopy $\gamma''_i$ such that $\gamma''_i|_{D\times J_k}$ realizes $\gamma_i'|_{D \times [\tau_{k-1},\tau_k]}$.  The resulting maps $f_i''$ and $\gamma_i''$ are disjoint homotopies.  Note that $diam \left(\gamma'_i \left(\{x\} \times I \right) \right) < \epsilon/2$ since $\gamma'_i$ is an $\epsilon/4$-approximation of a constant homotopy.  Thus $\gamma''$ is an $\epsilon/2$-approximation of $\gamma'$.   For $e=0,1$, note that $(f'_i)_e = (f''_i)_e$ and $(\gamma'_i)_e = (\gamma''_i)_e$.  By choice of $\zeta$, $f_i$, $\gamma_i$, and our approximations of $f_i$ and $\gamma_i$,  $(f''_i)_e$ is $\xi$-homotopic to $(f_i)_e$ and $(\gamma''_i)_e$ is $\xi$-homotopic to $(\gamma_i)_e$.  Note that these $\xi$-homotopies necessarily have disjoint images by virtue of our choice of $\xi$.  By adjoining these $\xi$-homotopies to the ends of $f''_i$ and $\gamma''_i$ and reparameterizing  appropriately, adjusting the parameter values only very near the ends of the original homotopies $f''_i$ and $\gamma''_i$ to cover the adjoined homotopies, we may assume without loss of generality that $(f''_i)_e = (f_i)_e$ and $(\gamma''_i)_e = (\gamma_i)_e$, still maintaining that $f''_i$ and $\gamma''_i$ are $\epsilon/2$-approximations of $f'_i$ and $\gamma'_i$, respectively, and insuring that $f''_i$ and $\gamma''_i$ are disjoint homotopies.

Observe that final adjusted maps $f_i''$ and $\gamma_i''$ are disjoint homotopies that are $3\epsilon/4$-approximations of $f_i$ and $\gamma_i$, respectively.     We likewise obtain $g''_i$ and $\lambda''_i$ that are disjoint homotopies and $3\epsilon/4$-approximation of $g_i$ and  $\lambda_i$, respectively.  Now we form the adjunction $f'= f''_1\cdot\lambda''_1 \cdot f''_2 \cdot \lambda''_2 \cdot \ldots f''_{2^N} \cdot \lambda''_{2^N}$ and $g''= \gamma''_1\cdot g''_1 \cdot \gamma ''_2 \cdot g''_2 \cdot \ldots \gamma''_{2^N} \cdot g''_{2^N}$, which are $3\epsilon/4$-approximations of $\tilde{f}$ and $\tilde{g}$, respectively.  Therefore $f'$ and $g'$ are the desired $\epsilon$-approximations of $f$ and $g$, respectively, that are disjoint homotopies.
\end{proof}

\begin{rmk} Note that in the proof above, there is  no need for the $\delta$-control omitted by the modified version of the definition of $\delta$-fractured maps in the reparamaterization  of $\gamma_i'$ to obtain $\gamma''_i$, as size controls are maintained by virtue of the homotopies being thin.
\end{rmk}

\begin{cor}[{\bf Modified Version of Theorem 3.8}] \label{Mod FM thm}
Every resolvable generalized manifold having the $\delta$-FMP satisfies the DHP, and therefore is a codimension one manifold factor.
\end{cor}

The maps of $f,g:D^2 \to X$ are said to be \emph{$0$-stitched} provided that there are $0$-dimensional $F_\sigma$ sets $A$ and $B$ contained in the interior of $D^2$ such that $f(D^2-A) \cap g(D^2-B) = \emptyset$.  We say that $f$ and $g$ are \emph{$0$-stitched along $A$ and $B$}.  If $Y$ and $Z$ are sets in $D^2$ missing $A$ and $B$ respectively, then we say that $f$ and
$g$ are \emph{$0$-stitched away from $Y$ and $Z$}. An \emph{infinite $1$-skeleton of $D^2$}, denoted $(K^\infty)^{(1)}$, is defined by $(K^\infty)^{(1)} = \bigcup K_i^{(1)}$, where $\{K_i\}$ is a sequence of triangulations of $D^2$ such that $K_1<K_2< \ldots$ and $mesh(K_i) \to 0$.

\begin{defn}\label{0SD defn}
A space $X$ has the \emph{$0$-stitched disks property} if any two maps $f,g:D^2 \to X$ can be approximated arbitrarily closely  by maps $f',g':D^2 \to X$ such that $f'$ and $g'$ are $0$-stitched along $0$-dimensional $F_\sigma$-sets $A$ and $B$ and away from infinite $1$-skeleta $(K_j^\infty)^{(1)}$, $j=1,2$, of $D^2$  such that $f'|_{(K_1^\infty)^{(1)}} \cup g'|_{(K_2^\infty)^{(1)}}$ is $1-1$.
\end{defn}

\begin{thm} \label{0SD thm} \cite{Halverson 3}
Every resolvable generalized manifold having the $0$-stitched disks property satisfies the DHP, and hence is a codimension one manifold factor.
\end{thm}

\noindent For the purposes of this paper, we define the following:

\begin{defn}
A space $X$ has the \emph{closed $0$-stitched disks property} if it has the $0$-stitched disks property where ``$F_\sigma$-sets $A$ and $B$'' is replaced with  ``closed sets $A$ and $B$'' in Definition \ref{0SD defn}.
\end{defn}

\noindent Clearly, if a space has the closed $0$-stitched disks property, then it has the $0$-stitched disks property.  Thus the following is an immediate corollary to Theorem \ref{0SD thm}.

\begin{cor} \label{C0SD}
Every resolvable generalized manifold having the closed $0$-stitched disks property satisfies the DHP, and hence is a codimension one manifold factor.
\end{cor}

\subsection{The Disjoint Topographies Property and Related Properties}

A characterization of codimension one manifold factors can be stated in terms of path concordances.  A {\em path concordance} in a space $X$ is a map $F:D \times I \to X \times I$ such that $F(D \times e) \subset X \times e, e \in \{0,1\}.$ Let $proj_X: X \times I \to X$ denote the natural projection map.

\begin{defn}
A metric space ($X,\rho$) satisfies the {\em Disjoint Path
Concordances Property (DCP)} if, for any two path homotopies
$f_i:D \times I \to X$ ($i=1,2$) and any $\epsilon > 0$, there
exist path concordances $F_i: D \times I \to X \times I$ such that
\begin{center}
$F_1(D \times I) \cap F_2(D \times I) = \emptyset$
\end{center}
and $\rho (f_i, {proj}_X \circ F_i) < \epsilon$.
\end{defn}

\begin{thm} \cite{Daverman-Halverson 2}
A resolvable generalized manifold is a codimension one manifold factor if and only if it has the DCP.
\end{thm}

An equivalent characterization of codimension one manifold factors, motivated by viewing the disjoint path concordances property with respect to the projections of the concordances to the parameter space $I$, can be formulated in the realm of topographies.  A \emph{topography $\Upsilon$ on $Z$} is a partition of $Z$ induced by a map $\tau: Z \to I$.  The \emph{$t$-level of $\Upsilon$} is given by $$ \Upsilon_t = \tau^{-1}(t).$$ A \emph{topographical map pair} is an ordered pair of maps $(f,\tau)$ such that $f:Z \to X$ and $\tau: Z \to I$.  The topography associated with $(f, \tau)$ is $\Upsilon$, where $ \Upsilon_t = \tau^{-1}(t).$  Suppose that for $i=1,2$, $\Upsilon^i$ is a topography on $Z$ induced by $\tau_i$ and $f_i: Z \to X$.  Then $(f_1, \tau_1)$ and $(f_2, \tau_2)$ are \emph{disjoint topographical map pairs} provided that for all $t \in I$, $$f_1( \Upsilon^1_t) \cap f_2(\Upsilon^2_t) = \emptyset.$$

\begin{defn}
A space $X$ has the \emph{disjoint topographies property (DTP)} if any two topographical map pairs $(f_i, \tau_i)$ ($i=1,2$), where $f_i: D^2 \to X$, can be approximated arbitrarily closely by disjoint topographical map pairs.
\end{defn}

\begin{thm} \cite{HaRe2}
Every resolvable generalized manifold is a codimension one manifold factor if and only if it  has the DTP.
\end{thm}

Although a powerful tool, providing an actual characterization of codimension one manifold factors, the disjoint topographies property (and the disjoint path concordances property) is generally accessed through other weaker forms of general position properties.

\begin{defn} \label{CRP def}
A generalized $n$-manifold $X$ has the \emph{crinkled ribbons property (CRP)} provided that any constant homotopy $f: K \times I \to X$, where $K$ is a $1$-complex can be approximated arbitrarily closely by a map $f':K \times I \to X$ so that:
\begin{enumerate}
\item $f'(K \times \{0\}) \cap f'(K \times \{1\}) = \emptyset$; and
\item $dim(f'(K \times I))\leq n-2$.
\end{enumerate}
\end{defn}

\begin{thm} \cite{HaRe2}
Every resolvable generalized $n$-manifold, $n \geq 4$, with the crinkled ribbons property has the DTP, and is therefore a codimension one manifold factor.
\end{thm}

\begin{defn} \label{TCRP def}
A generalized $n$-manifold $X$ has the \emph{twisted crinkled ribbons property (CRP-T)} provided that any constant homotopy $f: D \times I$ can be approximated arbitrarily closely by a map $f':D \times I$ so that:
\begin{enumerate}
\item $f'(D \times \{0\} ) \cap f'(D \times \{1 \} )$ is a finite set of points; and
\item $dim(f'(D \times I))\leq n-2$.
\end{enumerate}
\end{defn}

\begin{thm} \cite{HaRe2}
Every resolvable generalized $n$-manifold of dimension $n \geq 4$ having  the twisted crinkled ribbons property and the property that points
are $1$-LCC embedded in $X$ has the DTP, and is therefore a codimension one manifold factor.
\end{thm}

\begin{rmk}
In both Definitions \ref{CRP def} and \ref{TCRP def}, the condition
\begin{quote}
(2) $dim{(im(f'))}\leq n-2$.
\end{quote}
may be replaced with
\begin{quote}
(2*) $im(f')$ is 0-LCC embedded in $X$ with empty interior.
\end{quote}
This follows from a result by Borel (see Proposition 4.9 in \cite{Borel}) stating that if $X$ is a cohomological $n$-manifold and $Z$ is a closed subset of $X$, then $dim(Z) \leq n-2$ if and only if $Z$ has empty interior and is 0-LCC embedded in $X$.
\end{rmk}

A topographical map pair $(f, \tau)$ is said to be in the \emph{$\mathcal{K}$ category} if for some $1$-complex $K$, $K \times I$ is the domain of $f$ and $\tau$, and $f:K \times I \to X$ so that $K \times \{e\} \subset \tau^{-1}(e)$ for $e=0,1$.  In this case, we shall denote $(f, \tau) \in \mathcal{K}$.  A topographical map pair $(f,\tau) \in \mathcal{K}$ is said to be in the \emph{$\mathcal{K}_c$ category} if $f:K \times I \to X$ is a constant homotopy on $K$ and $\tau: K \times I \to I$ such that $\tau(x,t)=t$.

\begin{defn}
Let $(f_i, \tau_i) \in \mathcal{K}$ be such that $f_i: K_i \times I \to X$ and $\tau_i: K_i \times I \to I$.  Then $(f_1, \tau_1)$ is said to be {\it fractured} over a topographical map pair $(f_2, \tau_2)$ if there are disjoint balls $B_1, B_2, \ldots, B_m$ in $K_1 \times I$ such that:
    \begin{enumerate}
        \item $f_1^{-1}(im(f_2)) \subset \bigcup_{j=1}^m int(B_j)$; and
        \item $\tau_2 \circ f_2^{-1}\circ f_1(B_i) \ne I$.
    \end{enumerate}
\end{defn}

\begin{defn}\label{FRP def}
A space $X$ is said to have the \emph{fuzzy ribbons property (FRP)} provided that for any topographical map pairs, $(f_1, \tau_1) \in \mathcal{K}$ and $(f_2, \tau_2) \in
\mathcal{K}_c$, and $\epsilon > 0$ there are topographical map pairs $(f'_i, \tau'_i) \in \mathcal{K}$, for $i=1,2$,  such that $f'_i$ is an $\epsilon$-approximation  of  $f_i$  and $(f'_1, \tau'_1)$ is fractured over $(f'_2, \tau'_2)$.
\end{defn}

\begin{thm} \label{FRP thm}  \cite{HaRe2} Every resolvable generalized manifold having the FRP has the DTP, and is therefore is a codimension
one manifold factor.  \end{thm}

\begin{defn} A space $X$ is said to have the \emph{strong fuzzy ribbons property (FRP*)} provided that it satisfies the conditions of FRP in Definition \ref{FRP def}, where $\tau'_2: K_2 \times I \to I$ is specified to be the natural projection map, i.e., $\tau'_2(x,t)=t$.
\end{defn}

Clearly, a space that has the FRP* also has the FRP.  Thus we have the following corollary to Theorem \ref{FRP thm}.

\begin{thm} \label{FRP cor} Every resolvable generalized manifold having the FRP* has the DTP, and is therefore is a codimension
one manifold factor.  \end{thm}

\section{The Piecewise Disjoint arc-disk Property}

In this section we introduce the piecewise disjoint arc-disk property and the piecewise disjoint arc-disk property*.  We prove the main results associated with these properties.

\begin{defn}
A space $X$ is said to have the \emph{piecewise disjoint arc-disk property (P-DADP)} if for every $f:D^2 \to X$, $\alpha: I \to X$, and  $\epsilon > 0$ there is a cell complex $T$ of $D^2$ and approximations $f':D^2 \to X$ and $\alpha': I \to X-f'(T^{(1)})$ so that for each $\sigma \in T^2$, there is an $\epsilon$-homotopy $H_{\sigma} : I \times [0,1] \to X-f'(T^{(1)})$ from $\alpha'$ to a map $\alpha'': I \to X-f'(\sigma)$.
\end{defn}

\begin{defn}
A space $X$ is said to have the \emph{piecewise disjoint arc-disk property* (P-DADP*)} if for every $f:D^2 \to X$, $\alpha: L \to X$ where $L$ is a $1$-complex, and  $\epsilon > 0$ there is a cell complex $T$ of $D^2$ and approximations $f':D^2 \to X$ and $\alpha': L \to X-f'(T^{(1)})$ so that for each $\sigma \in T^2$, there is an $\epsilon$-homotopy $H_{\sigma} : L \times [0,1] \to X-f'(T^{(1)})$ from $\alpha'$ to a map $\alpha'': L \to X-f'(\sigma)$.
\end{defn}

Notice that the P-DADP (and the P-DADP*) does not necessarily imply the DADP.  The image of $\alpha''$ may still necessarily intersect $f'(T-\sigma)$.  Another important note is that the requirement that the homotopy pushing $\alpha'$ off of $f'(\sigma)$ miss $f'(\partial \sigma) \subset f'(T^{(1)})$ necessarily requires dimension $n \geq 4$.  Obviously, this is not a property satisfied in a $3$-manifold.

One may wonder if this property is too much to hope for in a space that does not have the DADP.  However, the P-DADP is satisfied in every example presently known to the authors of a codimension one manifold factor of dimension $n\geq 4$ detected by general position properties.  This is because the P-DADP is implied by other general position properties satisfied by these spaces.

\subsection{P-DADP Variations}

We begin with some preliminary results establishing connections between the P-DADP, the P-DADP*, and variations of these properties.

\begin{lem} \label{arcs lem}
Let $X$ be a path connected ANR that has the P-DADP.  Then $X$ satisfies the following:  For every $f:D^2 \to X$, $\alpha_i: I \to X$, where $i=1,\ldots,m$, and  $\epsilon > 0$ there is a cell complex $T$ of $D^2$ and approximations $f':D^2 \to X$ and $\alpha'_i: I \to X-f'(T^{(1)})$ so that for each $\sigma \in T^2$, there is an $\epsilon$-homotopy $H^i_{\sigma} : I \times [0,1] \to X-f'(T^{(1)})$ from $\alpha'_i$ to a map $\alpha''_i: L \to
X-f'(\sigma)$.  Moreover, if there are closed sets $A_i \subset I$ such that $\alpha'_i(A_i) \cap f'(\sigma) = \emptyset$, then we may require $H^i_{\sigma}|_{A_i \times I}$ to be a constant homopty.
\end{lem}

\begin{proof}  Divide $I$ into $2m-1$ intervals.  Find arcs $\gamma_i$ from $\alpha_i(1)$ to $\alpha_{i+1}(0)$ for $i=1, \ldots, m-1$.  Now define a single arc $\beta: I \to X$ that is the adjunction of all of these arcs $\alpha_1 \cdot \gamma_1\cdot\alpha_2\cdot\ldots \cdot \alpha_{n-1}\cdot\gamma_{n-1}\cdot\alpha_n$ such that $$\beta(t) = \left\{
                    \begin{array}{ll}
                      \alpha_i\left((2m-1)t-(2i-2)\right), & \hbox{ if } t \in \left[ \frac{2i-2}{2m-1}, \frac{2i-1}{2m-1}\right]\\
                      \gamma_i\left((2m-1)t-(2i-1)\right), & \hbox{ if } t \in \left[ \frac{2i-1}{2m-1}, \frac{2i}{2m-1}\right]
                    \end{array}
                  \right. $$
Now apply the P-DADP to $f$ and $\beta$.  The desired maps for the arcs $\alpha_i$ are obtained be restricting the maps associated with $\beta$ to the proper subintervals.

The moreover part of the lemma follows from an application of the map extension property for ANR's.
\end{proof}

\begin{thm} \label{P-DADP equiv}
Let $X$ be a path connected ANR that has the P-DADP and the $(0,2)$-DDP.  Then $X$ has also the P-DADP*. \end{thm}

\begin{proof}
Suppose $X$ is a path connected ANR that has the P-DADP and the $(0,2)$-DDP.
Let $f:D^2 \to X$, $\alpha: L \to X$, where $L$ is a $1$-complex, and  $\epsilon > 0$.  By the $(0,2)$-DDP we may assume without loss of generality that $\alpha(L^{(0)}) \cap f(D^2) = \emptyset$. Let $\xi < \min\{\epsilon,\frac 12 dist \left( \alpha(L^{(0)}), f(D^2)\right)\}$.  Let $\{\kappa_1, \ldots,\kappa_m\}$ denote the $1$-simplices of $L$.   Let $\alpha_i: \kappa_i \to X$.  Let $\zeta>0$ be such that any $\zeta$-approximation of $\alpha_i$ is $\xi$-homotopic to $\alpha_i$, for any $i = 1, \ldots, m$.

By Lemma \ref{arcs lem}, there is a cell complex $T$ of $D^2$ and $\zeta$-approximations $f':D^2 \to X$ and $\alpha_i': \kappa_i \to X-f'(T^{(1)})$, such that for each $\sigma \in T^2$, there are $\epsilon$-homotopies $H^i_{\sigma} : \kappa_i \times [0,1] \to X-f'(T^{(1)})$ of $\alpha'_i$ to a map $\alpha''_i: \kappa_i\to X-f'(\sigma)$.   By our choice of $\zeta$, there is a $\xi$-homotopy $G_i: \kappa_i \times [0,1] \to X$ from $\alpha_i$ to $\alpha'_i$. By our choice of $\xi$, $G_i(\partial \kappa_i \times [0,1]) \cap f'(D^2) = \emptyset$.  Let $A_i$ be a closed neighborhood of $\partial \kappa_i$ in $\kappa_i$ such that $G_i(A_i \times [0,1]) \cap f'(D^2) = \emptyset$.  Now apply the moreover part of Lemma \ref{arcs lem} to require that $H^i_{\sigma}$ is a constant homotopy on $A_i$.

Now we define new maps $\tilde{\alpha}'_i:\kappa_i \to X$ such that $\tilde{\alpha}'_i|_{\partial \kappa_i} = {\alpha}_i|_{\partial \kappa_i}$  and  homotopies $\tilde{H}^i_{\sigma}:\kappa_i \times I \to X -f'(T^{(1)})$ taking $\tilde{\alpha}'$ to $\tilde{\alpha}'': \kappa_i \to X-f'(\sigma)$ as follows:  Recall that we have required that $(H^i_{\sigma})_t(x) = \alpha_i'(x)$ for all $x \in A_i$ and $t \in [0,1]$.  Let $\overline{G}_i$ be the reverse of $G_i$, taking $\alpha'_i$ back to $\alpha_i$, so that $(\overline{G}_i)_t = (G_i)_{1-t}$.  Taper $\overline{G}_i$ to get a $\xi$-homotopy  $\overline{G}^*_i$ that is the constant map $\alpha'_i$ on $\kappa_i - int(A_i)$, and ends at $\alpha_i$ on $\partial \kappa_i$.  Let $\tilde{\alpha}'_i = (\overline{G}_i^*)_1$.  Note that $\tilde{\alpha}'_i(A_i) \cap  f'(D^2)= \emptyset$ and $\tilde{\alpha}'_i: \kappa_i \to X - f'(T^{(1)})$.  Let $\tilde{H}^i_{\sigma}: \kappa_i \times [0,1] \to X - f'(T^{(1)})$ be the homotopy such that $$\tilde{H}^i_{\sigma}(x,t) = \left\{
        \begin{array}{ll}
          \tilde{\alpha}'_i(x), & \hbox{ if } x \in A_i\\
          H^i_{\sigma}(x,t), & \hbox{ otherwise.}
        \end{array}
      \right.$$
Note that $\tilde{\alpha}''_i = (\tilde{H}^i_{\sigma})_1: \kappa_i \to X - f'(\sigma)$.

Let $\alpha' = \bigcup \tilde{\alpha}'_i$ and $H_{\sigma} = \bigcup \tilde{H}^i_{\sigma}$.  Then for each $\sigma \in T^{(2)}$,  $H_{\sigma}: L\times I \to X-f'(T^{(1)})$ is an $\epsilon$-homotopy from $\alpha': L  \to X-f'(T^{(1)})$ to  $\alpha'' = \bigcup \tilde{\alpha}''_i: L  \to X-f'(D^2)$.  Therefore, the cell complex $T$ and the maps ${\alpha}'$, $f'$ and  ${H}_{\sigma}$ are the desired maps to demonstrate that $X$ has the DADP*.
\end{proof}

The following lemma provides a generalization of the P-DADP*  that will be needed in the proof of one of the key theorems.

\begin{lem} \label{L1}
An ANR $X$ has the P-DADP* if and only if it has the following property:

\begin{quote}
(\dag) For every $\epsilon >0$ and $f: K \times I \to X$ and $\alpha: L \to X$, where $K$ and $L$ are $1$-complexes, there exist a cell complex $T$ of $K \times I$ and approximations $f'$ of $f$ and $\alpha': L \to X-f'(T^{(1)})$ of $\alpha$ such that for each $\sigma
\in T^{(2)}$, there is an $\epsilon$-homotopy $H_{\sigma} : L \times [0,1]
\to X-f'(T^{(1)})$ from $\alpha'$ to a map $\alpha'': L \to X-f'(\sigma)$.
\end{quote}

\noindent Moreover, if there is a closed set $A \subset L$ such that $\alpha'(A) \cap f'(\sigma) = \emptyset$, then we may require $H_{\sigma}|_{A \times I}$ to be a constant homopty.
\end{lem}

\begin{proof}  It suffices to show the forward direction because the reverse direction is trivial.

Suppose that $X$ has the P-DADP*. Let $\epsilon >0$ and $f: K \times I \to X$, $\alpha: L \to X$, where $K$ and $L$ are $1$-complexes.   Let $\{\kappa_1, \kappa_2, \ldots, \kappa_m\}$ be the collection of $1$-simplices of $K$.  We will define the desired cell complex $T$ and maps $f'$ and $\alpha'$ inductively for $i=1,\ldots, m$.

By the P-DADP* and the map extension properties of ANR's, there is a cell complex $T_1$ of $\kappa_1 \times I$ and approximations $f'_1$ of $f$ and $\alpha'_1: L \to X -f'_1(T_1^{(1)})$ of $\alpha$ so that the conclusion of property (\dag) holds with respect to $T_1$, $f'_1|_{\kappa_1 \times I}$, and $\alpha'_1$.  Note that any sufficiently close approximation of $f'_1$ and $\alpha'_1$ also satisfies the same conditions.

Fix $i$ where $1\leq i < m$.  Suppose that $T_j$ for $1\leq j\leq i$, $f'_i$, and $\alpha'_i: L \to X - f'_i(T_1^{(1)} \cup \ldots \cup T_i^{(1)})$ have been defined so that the conclusion of property (\dag) holds with respect to $T_1 \cup \ldots \cup T_i$, $f'_i|_{(\cup_{j=1}^{i}\kappa_j) \times I}$, and $\alpha'_i$.  By the P-DADP* and the map extension properties of ANR's, there is a cell complex $T_{i+1}$ of $\kappa_{i+1} \times I$ and approximations of $f'_{i+1}$ of $f'_i$ and $\alpha'_{i+1}: L \to X -f'_{i+1}( T_{i+1}^{(1)})$ of $\alpha'_i$ so that the conclusion of property (\dag) holds with respect to $T_{i+1}$,  $f'_{i+1}|_{\kappa_{i+1} \times I}$ and $\alpha'_{i+1}$.  Moreover, we require the approximations to be sufficiently close  that it is also the case that $\alpha'_{i+1}: L \to X -f'_{i+1}(T_1^{(1)} \cup \ldots \cup T_{i}^{(1)})$ and property (\dag) holds with respect to $T_1 \cup \ldots \cup T_i$,  $f'_{i+1}|_{(\cup_{j=1}^{i}\kappa_j) \times I}$ and $\alpha'_{i+1}$.  Then property (\dag) holds with respect to $T_1 \cup \ldots \cup T_{i+1}$,  $f'_{i+1}|_{(\cup_{j=1}^{i+1}\kappa_j) \times I}$ and $\alpha'_{i+1}$.

 Let $T=T_1\cup \ldots \cup T_m$, $f'=f'_m$ and $\alpha'=\alpha'_m$.  Then $T$ is the desired cell complex and $f'$ and $\alpha'$ are the desired maps to conclude our proof in this direction.

The moreover part of the lemma follows from an application of the map extension property for ANR's.

\end{proof}

\subsection{The Proof of the Main Theorem}

We now aim to prove the main theorem and a $1$-complex analogue of the main theorem. The $1$-complex analogue will follow as a corollary of the following:

\begin{thm} \label{pre main}
Suppose that $X$ is an ANR.  Then $X$ satisfies the P-DADP* if and only if $X$ satisfies the FRP*.
\end{thm}

\begin{proof}
First we prove the forward direction.  Suppose that $X$ satisfies the P-DADP*. Let $(f_1, \tau_1) \in \mathcal{K}$, $(f_2, \tau_2) \in \mathcal{K}_c$, and $\epsilon> 0$, where $f_i: K_i \times I \to X$. Let $\alpha = f_2|_{K_2 \times \{0\}}$.  Apply the P-DADP* and Lemma \ref{L1} to obtain the promised cell complex $T$ of $K_1 \times I$ and approximations $f_1':K_1 \times I\to X$ of $f_1$ and $\alpha':K_2 \to X-f(T^{(1)})$ of $\alpha$ satisfying the condition of property (\dag) in Lemma \ref{L1}.  Let $\{\sigma_1, \ldots, \sigma_m\}$ denote the collection of all $2$-cells in $T$.  For each $j=1,\ldots,m$, let $H_{\sigma_j}: K_2 \times I \to X-f'_1(T^{(1)})$ be the promised homotopy that pushes $\alpha'$ to $\alpha''_j: K_2 \to X-f'_1(\sigma_j)$ and $H_{\sigma_j}^-$ be the reverse of $H_{\sigma}$ that pushes $\alpha''$ back to $\alpha'$.  Let $f'_2 = H_{\sigma_1}\cdot H_{\sigma_1}^- \cdot H_{\sigma_2}\cdot H_{\sigma_2}^-  \cdot \ldots \cdot H_{\sigma_m}\cdot H_{\sigma_m}^- $.  Note that $f_2': K_2 \times I \to X - f_1(T^{(1)})$.  Let $\tau_i': K_i \times I \to I$ be the natural projection map.

We claim that $(f'_1, \tau'_1)$ is fractured over $(f'_2,\tau'_2)$.  Since $f_1'(T^{(1)}) \cap f'_2(K_2 \times I)= \emptyset$, it follows that $f_1'(\partial \sigma_j) \cap f'_2(K_2 \times I)= \emptyset$ for each $\sigma_j \in T^{(2)}$. Thus there is a ball $B_j$ in the interior of each $\sigma_j \in T^{(2)}$   such that  $\sigma_j \cap (f_1')^{(-1)}(im(f'_2)) \subset B_j$. Hence $(f'_1)^{-1}(im(f'_2)) \subset \bigcup_{j=1}^m int(B_j)$.   This is the first condition that must be satisfied.  Also note that $(f'_2)_{\frac{2j-1}{2m}} = \alpha''_j$ and $f_1'(\sigma _j) \cap \alpha''_j(K_2) = \emptyset$.  Hence $f'_1(B_j) \cap (f'_2)_{\frac{2j-1}{2m}}(K_2) =\emptyset$.  Thus  $\tau'_2 \circ (f'_2)^{-1}\circ f'_1(B_j) \subset I-\{ \frac{2j-1}{2m} \}$ which means $\tau'_2 \circ (f'_2)^{-1}\circ f'_1(B_j) \ne I$.  This is the second condition that must be satisfied.  Hence $(f'_1, \tau'_1)$ is fractured over $(f'_2,\tau'_2)$.  Therefore $X$ satisfies the FRP*.

To prove the reverse direction, suppose that $X$ satisfies the FRP*. Let $f:D^2 \to X$ and $\alpha:L \to X$.  Let $f_1:D \times I \to X$ be the map $f$, where $D^2 = D \times I$ and  $f_2:L \times I \to X$ be the constant homotopy on $\alpha$.  Let $(f_i,\tau_i)$  be the topographical map pair on the homotopy $f_i$ such that $\tau_i: \kappa_i \times I \to I $ is the natural projection map.  By the FRP* there are topographical map pairs $(f'_i,\tau'_i) \in \mathcal{K}$ such that $f'_i$ is an $\epsilon$-approximation of  $f_i$,   $(f_1', \tau'_1)$ is fractured over $(f_2', \tau'_2)$, where $\tau'_2$ is the natural projection map.  This means that  there are disjoint balls $B_1, B_2, \ldots, B_m$ in $D
\times I$ such that:
    \begin{enumerate}
        \item $(f'_1)^{-1}(im(f'_2)) \subset \bigcup_{j=1}^m int(B_j)$; and
        \item $\tau'_2 \circ (f'_2)^{-1}\circ f'_1(B_j) \ne I$.
    \end{enumerate}
Without loss of generality we may assume that the balls $B_1, B_2, \ldots, B_m$ are subpolyhedra of $D \times I$.  Define a cell complex $T$  from a partition of $D \times I$ into a collection of $2$-cells so that $\{B_1, B_2, \ldots, B_m\}\subset T^{(2)}$.  Let $f'=f'_1$ and $\alpha'=(f_2')_0$.

Note that since $f'_1(T^{(1)}) \cap im(f'_2) = \emptyset$, it follows from condition (1) that $f'(T^{(1)}) \cap \alpha'(L) = \emptyset$.  Thus we get both $\alpha': L \to X-f'(T^{(1)})$ and  $f_2':L \times I \to X - f'(T^{(1)})$.

Let $\sigma \in T^{(2)}$.  If $\sigma \notin \{B_1, B_2, \ldots, B_m\}$, let $H_\sigma: L \times I \to X - f'(T^{(1)})$ be the constant homotopy on $\alpha'$.  Then $\alpha''=(H_\sigma)_1 = \alpha': L \times I \to X - f'(\sigma)$.  If $\sigma \in \{B_1, B_2, \ldots, B_m\}$, then let $j$ be the index such that $\sigma = B_j$.  Choose $t_j$ so that $\tau_1 \circ f_1^{-1}\circ f_2(B_j)$ misses $t_j$.  Let $H_\sigma: L \times I \to X - f'(T^{(1)})$ be the homotopy that realizes $f'_2|_{L \times [0,t_j]}$. Then $H_\sigma$ pushes $\alpha'$  to  $\alpha'' = (H_\sigma)_{t_j}: L \to X - f'(\sigma)$.

Therefore $T$ is the desired cell complex and $f'$ and $\alpha'$, together with the homotopies $H_\sigma$, are the desired maps which show that $X$ has the P-DADP*
\end{proof}

We can now derive the $1$-complex analogue of the main result as a corollary to Corollary \ref{FRP cor} and Theorem \ref{pre main}.

\begin{cor}
If $X$ is a resolvable generalized $n$-manifold that satisfies the P-DADP*, then $X \times \mathbb{R}$ is an $(n+1)$-manifold.
\end{cor}

\begin{thm} \label{FM}
Suppose $X$ is an ANR.  Then $X$ satisfies the $\delta$-FMP if and only if $X$ satisfies the P-DADP.
\end{thm}

\begin{proof}
The proof is exactly as that above, with $I$ playing the role of the $1$-complex $L$, except that the extra $\delta$-control needed for the forward direction is obtained by choosing $T$ with small mesh.
\end{proof}

We now note that the main result (Theorem \ref{main thm}) follows as a corollary to Corollary \ref{Mod FM thm} and Theorem \ref{FM}.

\medskip \noindent {\bf Theorem 1.1.} \emph{If $X$ is a resolvable generalized $n$-manifold that satisfies the P-DADP, then $X \times \mathbb{R}$ is an $(n+1)$-manifold.} \medskip

\section{Further Relationships}

In this section we demonstrate further relationships of the P-DADP and P-DADP* properties with other properties, as well as relationships amongst other properties that have not been previously addressed.

\begin{thm} \label{P2MP}
Every resolvable generalized manifold of dimension $n \geq 4$ with the P2MP has also the P-DADP.
\end{thm}

\begin{proof}
Let $X$ be an ANR of dimension $n \geq 4$ with the plentiful $2$-manifolds property.  Let $\alpha: I \to X$ and $f:D^2 \to X$.

Let $\alpha':I \to N \subset X$ be an approximation of $\alpha$ where $N$ is a $2$-manifold containing the image of $\alpha'$ in its interior.  By Lemma \ref{P2MP lem}, for which we will provide the proof below, we may assume without loss generality that $\alpha'$ is embedded in $N$.  Since any embedded arc in a $2$-manifold is tame, the image of $\alpha'$ can be collared in $N$ thereby providing an $\epsilon$-isotopy $g: I \times I \to N$ so that $g_0 = \alpha'$ and $g_0(I) \cap g_1(I) = \emptyset$.

Since $dim(N)=2$, it follows that $N$ is $0$-LCC embedded in $X$ (see Corollary 26.2A in \cite{Daverman book}). Thus we may approximate $f$ by  $f':D^2 \to X$ so that $f'(\mathbb{Q}^* \times \mathbb{Q}^*) \cap N = \emptyset$, where the $\mathbb{Q}^* = \mathbb{Q} \cap I$.

Let $T$ be the cell complex of $D^2$ so that $$T^{(2)} = \left\{ \left[\frac{i-1}{2^m}, \frac{i}{2^m} \right] \times \left[\frac{j-1}{2^m}, \frac{j}{2^m} \right] \ | \ 1\leq i,j \leq 2^m \right\}$$ where $m$ is sufficiently large so that if $\sigma \in T^{(2)}$, then $diam(f'(\sigma))<\delta= dist(\left(g_0(I), g_1(I)\right)$.    Hence for each $\sigma \in T^{(2)}$, the image $f'(\sigma)$ can meet only one of $g_0(I)$ or $g_1(I)$.  If $f'(\sigma)$ misses $g_0(I)$, then let $H_{\sigma}$ be the constant homotopy on $\alpha'=g_0$.  If $f'(\sigma)$ meets $g_0(I)$, then let $H_{\sigma}$ be the realization of $g$.

The cell complex $T$ and the maps $f'$ and $\alpha'$ together with the homotopies $H_\sigma$ demonstrate that $X$ has the P-DADP.
\end{proof}

\begin{lem} \label{P2MP lem}
If $X$ is a resolvable generalized manifold with the DAP and the P2MP and $\alpha: I \to X$, then $\alpha$ can be approximated arbitrarily closely by an embedding $\alpha': I \to int(N) \subset N \subset X$, where $N$ is a $2$-manifold.
\end{lem}

\begin{proof}  Since $X$ has the DAP, we may assume without loss of generality that $\alpha$ is an embedding in $X$.  Let $\epsilon>0$ be given.  By the continuity of $\alpha$, there is a $\delta >0$ such that whenever $A \subset I$ and $diam(A)< \delta$, we have $diam(\alpha(A)) < \epsilon/3$.  By the continuity of $\alpha^{-1}$, there is a $\gamma $ such that $0 < \gamma < \epsilon/3$ and whenever $Z \subset X$ and $diam(Z)< 2\gamma$, we have $diam(\alpha^{-1}(A)) < \delta$.  Let $\alpha'$ be a $\gamma$-approximation of $\alpha$ such that $\alpha': I \to int(N) \subset N \subset X$, which is promised by the P2MP.  Without loss of generality we may assume that $\alpha'$ is a piecewise linear map in general position in $N$.  Hence any self intersection points come in pairs.  Let $\{(t_1,t_1'),(t_2,t_2'), \ldots, (t_r,t_r')\}$ denote the pairs of values in $I$ such that $\alpha'(t_i) = \alpha'(t_i')$, for every $i=1, \ldots, r$.  Without loss of generality we may assume that $t_i < t_i'$ and $t_1 < t_2 < \ldots < t_r$.  Let $\tau_1=t_1$ and $\tau_1'=t_1'$.  Suppose $\tau_i$ has been defined and $\{ t_j \ |  \ \tau_i < t_j\} \ne \emptyset$.  Let $\tau_{i+1} = \underset{j\in \{1,\ldots,r\}}{min} \{ t_j \ | \ \tau'_i < t_j\}$.  Let $\tau_i' = t_j'$ where $\tau_i = t_j$.

Now define $\alpha'': I \to int(N) \subset N \subset X$ such that $\alpha''([\tau_i,\tau_i']) = \alpha''(\tau_i)=\alpha'(\tau_i')$ and $\alpha'(t) = \alpha'(t)$ otherwise, i.e., $t$ is not contained in an interval $[\tau_i,\tau_i']$.

Note that since $\alpha'(\tau_i) = \alpha'(\tau_i')$ and $\alpha'$ is a $\gamma$-approximation of $\alpha$, it follows that $dist(\alpha(\tau_i), \alpha(\tau'_i)) < 2 \gamma$. Hence $diam([\tau_i,\tau_i']) < \delta$.  Thus, $diam(\alpha([\tau_i,\tau_i'])) < \epsilon/3$. Hence, $diam(\alpha'([\tau_i,\tau_i'])) < \epsilon/3 + 2 \gamma < \epsilon$. Thus $\alpha''$ is an $\epsilon$-approximation of $\alpha'$.

A slight reparametrization of $\alpha''$ gives the desired embedding.
\end{proof}

Note that Theorem \ref{P2MP} applies to many types of spaces, including spaces of dimension $n\geq4$ that arise as a nested defining sequence of thickened $(n-2)$-manifolds, including the totally wild flow and the $k$-ghastly spaces constructed in \cite{Cannon-Daverman2,Halverson 1}.

\begin{thm}\label{CRP thm}
Every resolvable generalized manifold that has the CRP has also the P-DADP*.
\end{thm}

\begin{proof}  The argument is very similar to that in the proof of Theorem \ref{P2MP}. We again let $g:K \times I \to X$ be the constant homotopy on $\alpha:K \to X$, where $K$ is a $1$-complex.  By applying the crinkled ribbons property we obtain an approximation $g'$ such that $g'_0(K) \cap g'_1(K) = \emptyset$ and $dim(g'(K \times I))\leq n-2$.  The point set $g'(K \times I)$ now takes on the role of $N$ in the proof of Theorem \ref{P2MP} in terms of approximating $f$.  In particular, Corollary 26.2A of \cite{Daverman book} also implies that $g'(D \times I)$ is $0$-LCC in $X$ so we can find an approximation $f'$ of $f$ such that $f'(\mathbb{Q}^* \times \mathbb{Q}^*) \cap g'(D\times I) = \emptyset$. \end{proof}

\begin{thm}\label{TCRP thm}
Every resolvable generalized manifold that has $1$-LCC embedded points and the CRP-T has also the P-DADP.
\end{thm}

\begin{proof}  The argument is almost identical to that of the proof of Theorem \ref{CRP thm}, except there are a finite number of points of intersection of $g_0$ and $g_1$.  Specify that the image of $f$ misses these points.  Even more, let $W$ be a closed neighborhood of these points that misses the image of $f$.  Specify that the image of the approximation $f'$ of $f$ also misses $W$.  The only other modification is to let  $\delta$ be the distance between $\overline{g'_0(I)-W}$ and $\overline{g'_1(I)-W}$.  Then proceed as before.  \end{proof}

Since the $1$-LCC condition implies the $(0,2)$-DDP, the following result is a corollary of Theorems \ref{P-DADP equiv} and \ref{TCRP thm}.

\begin{cor}\label{TCRP cor}
Every resolvable generalized manifold that has $1$-LCC embedded points and the CRP-T has also the P-DADP*.
\end{cor}

Examples of spaces having the crinkled ribbons property are the locally spherical resolvable generalized $n$-manifolds, $n\geq 4$ (see \cite{HaRe2}).

\begin{thm} \label{DADP-CRP}
If $X$ is an ANR with the DADP, then $X$ has also the CRP.
\end{thm}

\begin{proof}
Let $f: K \times I \to X$, where $K$ is a $1$-complex, be a constant homotopy. Since $X$ has the DADP, it has also the DAP.  Thus, we can find an approximation $f'$ of $f$ so that  $f'(K \times \{0\}) \cap f'(K \times \{1\}) = \emptyset$. Applying the DADP we can then approximate $f'$ by a map $f''$ which misses the image of a countable dense collection of arcs in $X$.    The approximation should be sufficiently small to ensure $f''(K \times \{0\}) \cap f''(K \times \{1\}) = \emptyset$. Then $f''(K \times I)$ is $0$-LCC embedded in $X$ with empty interior. It follows that $dim(f''(K \times I))\leq n-2$  (see Proposition 4.9 in \cite{Borel}).  Hence $f''$ is the desired approximation of $f$.  Therefore $X$ has the CRP.
\end{proof}

\begin{thm} \label{DADP-TCRP}
If $X$ is an ANR with the DADP, then $X$ has also the CRP-T.
\end{thm}

\begin{proof}
By Theorem \ref{DADP-CRP}, the DADP implies the CRP.  The fact that the CRP implies the CRP-T immediately follows from the definitions.
\end{proof}

\begin{thm} \label{DADP-0SD}
If $X$ is an ANR with the DADP, then $X$ has also the closed $0$-stitched disks property.
\end{thm}

\begin{proof}
Let $f,g:D^2 \to I$.  Let $K_1<K_2<\ldots $ be a triangulation of $D^2$ such that $mesh(K_i) \to 0$.  Let $(K^\infty)^{(1)} = \bigcup K_i^{(1)}$.  Apply the DADP to get approximations $f'$ of $f$ and $g'$ of $g$ so that  $f'(K^\infty)^{(1)})\cap g'(D^2) = \emptyset$, $g'(K^\infty)^{(1)})\cap f'(D^2) = \emptyset$, and both $f'$ and $g'$ are 1-1 on $(K^\infty)^{(1)}$.  Let $A=(f')^{-1}(g'(D^2))$ and $B=(g')^{-1}(f'(D^2))$.  Note that $A$ and $B$ are closed $0$-dimensional sets contained in $D^2-(K^{\infty})^{(1)}$ and $f'(D^2-A) \cap g'(D^2-B) = \emptyset$.  In particular, $f'$ and $g'$ are $0$-stitched along closed sets
$A$ and $B$, and away from $(K^\infty)^{(1)}$ in $D^2$,  such that
$f'|_{(K^\infty)^{(1)}} \cup g'|_{(K^\infty)^{(1)}}$ is $1-1$.  Therefore $X$ as the closed $0$-stitched disks property.
\end{proof}

\begin{thm}\label{0SD-DFM}
Every ANR that has the closed $0$-stitched disks property has also the P-DADP.
\end{thm}

\begin{proof}
Let $f: D^2 \to X$ and $\alpha: I \to X$.  Let $g:D\times I \to X$ be the constant homotopy so that $g_t(x) = \alpha(x)$.  Apply the closed $0$-stitched disks property to obtain closed $0$-dimensional sets $A$ and $B$, infinite 1-skeleta $(K_j^\infty)^{(1)}$ for $j=1,2$, and approximations $f'$ and $g'$ such that $f'$ is 1-1 on $(K_1^\infty)^{(1)}$, $g'$ is 1-1 on $(K_2^\infty)^{(1)}$, and $f'(D^2-A) \cap g'(D^2-B) = \emptyset$.

Let $p: D\times I \to I$ be the natural projection map.  Without loss of generality we may assume that $p|_B$ is 1-1 (see Proposition 4.6 of \cite{Halverson 1}).  Let $J=[a,b]$ be an interval in $I$ so that $P(B) \cap J = \emptyset$.  Thus $B \cap (D\times J) = \emptyset$.  Since $g'$ is 1-1 on $(K_2^\infty)^{(1)}$,  we may also assume, by a slight modification of the level lines if necessary, that $D \times \{a,b\}$ is contained in the infinite $1$-skeleton so that $g'|_{D \times \{a,b\}}$ is 1-1.  More particulary, we need that $g'_a(D) \cap g'_b(D) = \emptyset$.   Let $\delta=dist(g'_a(D),g'_b(D))$. Let $\alpha' = g_a$.

By choice of $f'$ and $g'$, $f'(D^2-A) \cap g'(D^2-B) = \emptyset$.  Since $ D\times J \subset D^2-B$,  it follows that $(f')^{(-1)}g'(D\times J) \subset A$.  Let $\gamma>0$ be a value so that if $Z \subset D^2$ and $diam(Z)< \gamma$, then  $diam(f(Z))< \delta$.  Let $(K_1^{\infty})^{(1)} = \bigcup K_i^{(1)}$ where $K_1< K_2 < \ldots $ and $mesh(K_i) \to 0$.  Choose $m$ so that $mesh(K_m) < \gamma$.  The complex $K_m$ will be the required cell complex $T$ in the definition of the P-DADP.  Note that $f'(K_m^{(1)}) \cap g'(D\times J) = \emptyset$ since $f'^{(-1)}g'(D\times J) \subset A$ and $K_m^{(1)} \subset D^2 - A$.  Also note that $\alpha'=g_a:I \to X - f'(K_m^{(1)})$.  By choice of $\gamma$ and $i$, for each $\sigma \in K_m^{(2)}$, at least one of the cases $f'(\sigma) \cap g'(D \times \{a\})= \emptyset$ or $f'(\sigma) \cap g'(D \times \{b\})= \emptyset$ holds true.  If $f'(\sigma) \cap g'(D \times \{a\})= \emptyset$, let $H_\sigma$ be the constant path homotopy on $g'_a$.  If $f'(\sigma) \cap g'(D \times \{a\}) \ne \emptyset$, then $f'(\sigma) \cap g'(D\times \{b\}) = \emptyset$, so let $H_\sigma$ be the realization of $g'|_{D \times [a,b]}$.  The cell complex $T=K_i$ and the maps $f'$ and $\alpha'$ together with the homotopies $H_\sigma$ for $\sigma \in T^{(2)}$ demonstrate that $X$ has the P-DADP.
\end{proof}

The following is a  corollary to Theorems \ref{FM}, \ref{DADP-0SD}, and \ref{0SD-DFM}.  It also follows more directly as a corollary to Theorem \ref{FM} since the DADP implies the P-DADP trivially.

\begin{cor} \label{DADP FRP}
If $X$ is an ANR with the DADP, then $X$ has the $\delta$-FRP.
\end{cor}

\section{Summary of Property Relationships}

The relationships amongst general position properties used to detect codimension one manifold factors is summarized in the chart in Figure \ref{chart}.  Equivalent properties are boxed together.  The characterizing properties are indicated in the bolded box.  An arrow implies an implication.  A filled in dot at the beginning of the arrow  indicates that the reverse implication is known to be false.  An arrow without a filled in dot at the beginning indicates that the validity of the reverse implication is at present unknown.  A filled in dot with a line not ending in an arrow indicates that the implication is not known, but the reverse implication is known to be false.

\begin{figure}
    \begin{center}
\hspace{-.15 in} \epsfig{file=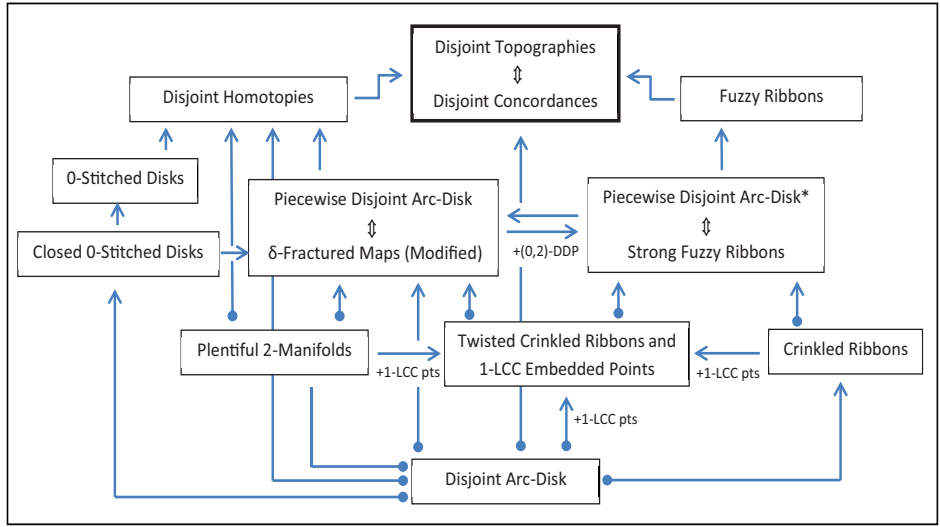, width=5.3 in, height=7.4 in} \vspace{-4 in}
\caption{Relationships amongst general position properties that detect codimension one manifold factors.} \label{chart}
    \end{center}
\end{figure}

The fact that a resolvable generalized manifold of dimension $n\geq 4$ is a codimension one manifold factor if and only if it has the disjoint path concordances property is proved in \cite{Daverman-Halverson 2}.  Thus the disjoint path concordances property provides a characterization of codimension one manifold factors.

The equivalence of the disjoint topographies property and the disjoint path concordances property was proved in \cite{HaRe2}.  Thus the disjoint topographies property also provides a characterization of codimension one manifold factors.

The fact that a resolvable generalized manifold that has the disjoint homotopies property is a codimension one manifold factor is proved in \cite{Halverson 1}.  It is unknown whether the disjoint topographies property implies the disjoint homotopies property.

The fact that a resolvable generalized manifold that has the piecewise disjoint arc-disk property also has the disjoint homotopies property and hence is a codimension one manifold factor is Theorem \ref{main thm}.  It is unknown whether the disjoint homotopies property implies the piecewise disjoint arc-disk property.

The fact that a resolvable generalized manifold that has the $0$-stitched disks property satisfies the disjoint homotopies property, and hence is a codimension one manifold factor is proved in \cite{Halverson 3}.  It is unknown whether the disjoint homotopies property implies the $0$-stitched disks property.

The fact that the closed $0$-stitched disks property implies the $0$-stitched disks property trivially follows from the definitions.  Thus Corollary \ref{C0SD} states that a resolvable generalized manifold that has the closed $0$-stitched disks property satisfies the disjoint homotopies property, and hence is a codimension one manifold factor.   It is unknown whether the $0$-stitched disks property implies the closed $0$-stitched disks property.

The fact that the closed $0$-stitched disk property implies the piecewise disjoint arc-disk property is Theorem \ref{0SD-DFM}.  It is unknown whether the reverse implication is true.

The fact that a resolvable generalized manifold that has a dense collection of $\delta$-fractured maps also has the disjoint homotopies property, and hence is a codimension one manifold factor is proved in \cite{Halverson 2}.  The fact that a resolvable generalized manifold that has the $\delta$-fractured maps property, defined in terms of the modified version of the definition of $\delta$-fractured maps, also has the disjoint homotopies property, and hence is a codimension one manifold factor is Theorem \ref{FM thm}. It is unknown whether the reverse implications are true.

The equivalence of the piecewise disjoint arc-disk property and $\delta$-fractured maps property is stated in Theorem \ref{main thm}.

The equivalence of the piecewise disjoint arc-disk property* and the strong fuzzy ribbons property is stated in Theorem \ref{pre main}.

The fact that the strong fuzzy ribbons property implies the fuzzy ribbons property trivially follows from the definition.  It is unknown whether the reverse implication is true.

The fact that the fuzzy ribbons property implies the disjoint topographies property is shown in \cite{HaRe2}.  It is unknown whether the reverse implication is true.

The fact that the piecewise disjoint arc-disk property* implies the piecewise disjoint arc-disk property trivially follows from the definition.  The reverse implication in the case of a space with the $(0,2)$-DDP is stated in Theorem \ref{P-DADP equiv}. It is unknown whether the reverse implication is true in general.

The fact that a resolvable generalized manifold of dimension $n\geq 4$ that has the plentiful $2$-manifolds property has the disjoint homotopies property, and hence is a codimension one manifold factor is proved in \cite{Halverson 1}.  The reverse implication is not true.  The $2$-ghastly space shown in \cite{Daverman-Walsh} to be codimension one manifold factors do not have the plentiful $2$-manifolds property.

The fact that the plentiful $2$-manifolds property implies the piecewise disjoint arc-disk property is stated in Theorem \ref{P2MP}.  The fact that the crinkled ribbons property implies the piecewise disjoint arc-disk property* is stated in Theorem \ref{CRP thm}.  The fact that the twisted crinkled ribbons property in the case of an ANR with $1$-LCC embedded point implies both the piecewise disjoint arc-disk property and the piecewise disjoint arc-disk property* is stated in Theorem \ref{TCRP thm} and Corollary \ref{TCRP cor}, respectively.   None of the reverse implications for these properties are true.  The $2$-ghastly space shown in \cite{Daverman-Walsh} to be codimension one manifold factors do not have any of these properties.

The fact the that the crinkled ribbons property implies the twisted crinkled ribbons property is immediate from the definitions.  The fact the plentiful $2$-manifolds property implies the twisted crinkled ribbons property in the case of a resolvable generalized manifold of dimension $n \geq 4$ is also immediate from the definitions. It is unknown whether the reverse implications are true.

The fact that a resolvable generalized manifold that has the disjoint arc-disk property is a codimension one manifold factor was first proved in \cite{Daverman 1}.  In fact, the disjoint arc-disk property  implies the disjoint homotopies property \cite{Halverson 1}. The reverse implication is not true.  The totally wild flow has the disjoint homotopies property, and therefore is a codimension one manifold factor, but fails to have the disjoint arc-disk property \cite{Cannon-Daverman2, Halverson 1}.

The fact that the disjoint arc-disk property implies the crinkled ribbons property is stated in Theorem \ref{DADP-CRP}. The fact that the disjoint arc-disk property implies the twisted crinkled ribbons property in the case of an ANR is stated in Theorem \ref{DADP-TCRP}.  The fact that the disjoint arc-disk property implies the closed $0$-stitched disks property, and hence the $0$-stitched disks property is stated in Theorem \ref{DADP-0SD}.  The fact that the disjoint arc-disk property implies the $\delta$-fractured maps property is Corollary \ref{DADP FRP}. This means that the disjoint arc-disk property implies all other properties, with the possible exception of the plentiful $2$-manifolds property.  The reverse implications are not true. There are $k$-ghastly spaces ($k>2$) that satisfy the plentiful $2$-manifolds property and the ribbons properties, but do not satisfy the disjoint arc-disk properties. There are $2$-ghastly spaces that satisfy all other properties besides the plentiful $2$-manifolds property, the crinkled ribbons property and the twisted crinkled ribbons property, but do not satisfy the disjoint arc-disk property.

\section{Epilogue}

We provide a list of several interesting problems that remain unsolved:

\begin{enumerate}
\item Does the P-DADP* imply the DTP?
\item Does the P-DADP imply the DHP?
\item Do the P-DADP or P-DADP* imply the $0$-stitched disks property?
\item Do the P-DADP or P-DADP* imply the closed $0$-stitched disks property?
\item Does the $0$-stitched disks property imply the P-DADP?
\item Does the DHP imply the $0$-stitched disks property?
\item Do $(n-2)$-dimensional decompositions arising from a defining
sequence of thickened $(n-2)$-manifolds have the P-DADP?
\item Recently, we have proved in \cite{HaRe3} that decomposition spaces resulting from decompositions of $\mathbb{R}^n$, ${n\geq 4}$, into convex sets are topologically equivalent to $\mathbb{R}^n$.  In fact, such spaces possess the DADP property.   Is there a generalization of this result utilizing the P-DADP?  For example, what about decompositions into star-like sets or sets that are homeomorphic to convex sets (such as decompositions into arcs and points)?
\item Does the P-DADP provide a characterization of resolvable generalized manifolds as codimension one manifold factors?
\item Do all resolvable generalized manifold of dimension $n\geq 4$ satisfy the P-DADP?
\end{enumerate}

In this paper we have demonstrated that, as a unifying property, the piecewise disjoint arc-disk property is a powerful tool in detecting codimension one manifold factors.  It has the potential to lead to even further insights in demonstrating that all resolvable generalized manifolds of dimension $n \geq 4$ are codimension one manifold factors or finding a counterexample, thereby solving the famous generalized R.L. Moore problem \cite{Daverman S, Bertinoro}.

\section*{Acknowledgments}

We were supported by the Slovenian Re\-search Agency grants
BI-US/11-12/023, P1-0292-0101 and J1-4144-0101, and research travel grants from Brigham Young University.
We thank the referees for several comments and suggestions.

\end{document}